\documentclass[12pt]{article}
\usepackage{amsfonts}
\usepackage{times}
\usepackage[left=1in,top=1in,right=1in]{geometry}
\usepackage{amsmath,amsthm,amssymb}
\usepackage{color}
\usepackage{latexsym}
\usepackage{cite}
\theoremstyle{plain}

\newtheorem{theorem}{Theorem}
\theoremstyle{plain}

\newtheorem{corollary}{Corollary}

\newtheorem{definition}{Definition}
\newtheorem{example}{Example}

\newtheorem{lemma}{Lemma}

\newtheorem{remark}{Remark}

\newcommand{\NKCMM}{$ N(\kappa) $-contact metric manifold}
\newcommand{\GTNWC}{generalized Tanaka-Webster connection}

\date{}

\begin{document}

\title{$ N(\kappa)- $contact Metric Manifolds with Generalized Tanaka-Webster
Connection}
\author{\textbf{ \.Inan \"Unal} \\     
%EndAName
{\normalsize Department of computer Engineering}\\{\normalsize  Munzur University, Turkey} \\{\normalsize inanunal@munzur.edu.tr}\\  \textbf{Mustafa Alt{\i}n}\\{\normalsize Technical Sciences Vocational School}\\{\normalsize Bing\"ol University, Turkey}\\ {\normalsize maltin@bingol.edu.tr}\\\ }
\maketitle

%\makenextpage

\noindent \textbf{Abstract:} In this paper we work
on  \NKCMM s\ with a \GTNWC\ . We obtain some curvature properties.
It is proven that if a \NKCMM\ with generalized
Tanaka-Webster connection is K-contact then it is an example of generalized Sasakian 
space form. Also we examine some flatness and symmetric conditions of concircular curvature tensor on a \NKCMM s with a generalized Tanaka-Webster connection. \newline 


\section{ Introduction}
A nullity condition for an almost contact metric manifold $ (M^{(2n+1)},\phi,\xi,\eta,g) $ was defined with curvature identity $ R(X_1,X_2)\xi=0  $ for all $ X_1,X_2 \in \Gamma(TM) $ by Blair et al. in \cite{blair1995contact}.  The tangent sphere bundle of a flat Riemannian manifold admits such a structure \cite{blair2010riemannian}. By apply  $ D- $homothetic deformations to this structure, a special class of contact manifolds is obtained. Such a manifold is called a $ (\kappa, \mu)- $space and it satisfies 
\begin{equation*}
R(X_1,X_2)\xi=(\kappa I+\mu h)(\eta(X_2)X_1-\eta(X_1)X_2)
\end{equation*}
where $ \kappa $ and $ \mu $ are constants and $ 2h $ is the Lie derivative of $ \xi $ in the direction $ \phi $. On the other hand $ (\kappa,\mu)- $ nullity distribution of a contact manifold is defined by 
\begin{equation*}
 N_{p}(\kappa,\mu )=\{X_{3}\in \Gamma	(T_{p}M):R(X_{1},X_{2})X_{3} =(\kappa I+\mu h)[g(X_{2},X_{3})X_{1}-g(X_{1},X_{3})X_{2}].
\end{equation*}
for all $ X_1,X_2 \in \Gamma(TM) $. If $ \xi \in N(\kappa,\mu) $ then the manifold is called $ (\kappa,\mu)- $contact metric manifold. On a $ (\kappa,\mu)- $contact metric manifold $ \kappa\leq 1 $ and if $ \kappa=1 , \mu=0 \ (\text{i.e $ \mu $ is indeterminant})$ then the manifold is to be Sasakian. Also it is known that for $ \kappa\leq 1 $ the $ (\kappa,\mu)- $ nullity  condition determines the curvature of $ M $ completely \cite{blair1995contact}. We get $ \kappa- $nullity distribution if $ \mu=0 $, such a structure was defined in \cite{tanno1988ricci}. $ \kappa- $nullity distribution of a Riemann manifold $ M $ is determined by 

\begin{equation*}
 N_{p}(\kappa)=\{X_{3}\in \Gamma
(T_{p}M):R(X_{1},X_{2})X_{3}=\kappa[g(X_{2},X_{3})X_{1}-g(X_{1},X_{3})X_{2}]\}.
\end{equation*}
for all $ X_1,X_2 \in \Gamma(TM) $. If $ \xi \in N(\kappa) $ then $ M $ is called $ N(\kappa)- $contact metric  manifold. This type of manifolds were studied in several context by many researchers \cite{blair2005concircular, de2014class,de2014bochner,barman2017n,majhi2015classifications,de2018certain,ingalahalli2019certain,ozgur2008n,de2010weyl}.  In \cite{blair1977two} Blair proved that an almost contact metric
manifold with the condition $R(X_{1},X_{2})\xi =0$ is locally the product of
a flat $(n+1)-$dimensional manifold and an $n-$dimensional manifold of
positive constant curvature $4.$ Thus it has been seen that a $N(0)-$contact
metric manifold is locally isometric to $\mathbb{S}^{n}(4)\times \mathbb{E}%
^{2n+1}.$\par 
In \cite{blair2005concircular} Blair et al. studied on \NKCMM s\  with concircular curvature tensor. They gave an example of  \NKCMM s. \  They proved that a \NKCMM\ is locally isometric to $ S^{(2n+1)}(1)$ under the condition $ \mathcal{Z}(X_1,\xi)\mathcal{Z}=0 $ or $\mathcal{Z}(X_1,\xi)R=0 $ for concircular curvature tensor $ \mathcal{Z}$, Riemann curvature tensor $ R $ and $ X_1\in \Gamma(TM) $. De et al. \cite{de2014class} worked on some flatness condition of concircular curvature tensor and they presented an example. Also they proved that a \NKCMM \ satisfies $ \mathcal{Z}(X_1,\xi)S=0 $ if and only if the manifold is an Einstein-Sasakian manifold. \par 
 A \GTNWC\  has been introduced by Tanno \cite{tanno1989variational} as a generalization of Tabaka-Webster connection \cite{tanaka1976non,webster1978pseudo}. Contact manifolds with \GTNWC\  were studied by many researchers \cite{de2016generalized,prakasha2018conharmonic,ghosh2017kenmotsu,montano2010some,kazan2018}. The concircular curvature tensor was defined by Yano \cite{yano1940concircular}. On the contact structures there many works under certain conditions of concircular curvature tensor \cite{atcceken2015almost,unal2018concircular,turgut2017conformal}. \par 
This paper is on \NKCMM s\ with \GTNWC. Firstly we give the definition of a \GTNWC\ for a \NKCMM. \ Then we obtained some basic results and curvature relation. We prove that a \NKCMM\  with \GTNWC \ is an example of generalized Sasakian space forms. Secondly  we consider the flatness conditions and some symmetry conditions of concircular curvature tensor related to \GTNWC \ on a \NKCMM.\

\section{Preliminaries}
In this section we give some basic facts about contact manifolds and \NKCMM s\ . For details we refer to reader \cite{blair2010riemannian}.\par 

Let $M$ be a $(2n+1)-$dimensional smooth manifold. $(\phi ,\xi ,\eta )$ is
called an almost contact structure on $ M $ if we have 
\begin{equation*}
\phi ^{2}X=X-\eta (X)\xi \ \ ,\ \phi (\xi )=0,\ \ \ ,\ \eta \circ \phi =0\ \
\ ,\ \eta (\xi )=1.
\end{equation*}
for a $(1,1)$ tensor field $\phi $,  a vector field $ 
\xi $ and a $1-$ form $\eta $  on $M$ .
The kernel of $\eta $ defines a non-integrable distribution on $M$ and the distribution is called by contact distribution. The rank of $\phi $ is $2m$. The Riemannian metric $ g $ is called associated metric if 
\begin{equation}
g(\phi X_{1},\phi
X_{2})=-g(X_{1},X_{2})+\eta (X_{1})\eta (X_{2})\label{g(phiX,phiY)}
\end{equation}
  and it is called compatible  metric if
  \begin{equation*}
  d\eta  (X_{1},X_{2})=g(\phi X_{1},X_{2})
  \end{equation*}
   for all $X_{1},X_{2}\in \Gamma (TM)$. The
manifold $M$ is called by almost contact metric manifold with the structure $%
(\phi ,\xi ,\eta )$ and associated metric $g$. \par 
The $ (1,1) -$tensor field $h=\frac{1}{2}\mathcal{L}_{\xi }\phi$ has an important role in the Riemannian geometry of contact manifolds, where $\mathcal{L}$ is the Lie derivative of $ \phi $ in the direction $ \xi $ . We have following properties for $ h $ \cite{blair2010riemannian};
\begin{lemma}
	On a contact metric manifold;
	\begin{enumerate}
		\item $ h $ is symmetric operator i.e $ g(hX,Y)=g(X,hY) $.
		\item The derivation of $ h $ in the direction $ \xi $ is given by
		\begin{equation}\label{nablaXksi}
		\nabla_X\xi=-\phi X-hX  
		\end{equation}
		\item $ h $ anticommutes with $ \phi $, i.e $ h\phi+\phi h=0 $.
		\item $ h $ is trace free.
		\end{enumerate}  

\end{lemma}
If characteristic vector field $ \xi $ is Killing vector field then $ M $ is called $ K- $contact. That is in a $ K- $contact manifold $ h=0 $.  An almost contact metric manifold $M$ is said to be normal if $ \phi $ is integrable, i.e Nijenhuis tensor of $ \phi $ vanishes. Also when the contact manifold is normal $ h=0 $ .
If the second fundamental form $ \Omega $ of an almost contact metric manifold $ M $ is $ \Omega(X_1,X_2 )=g(\phi X,Y)$ and $ M $ is normal then $ M $ is called Sasakian. A Sasakian manifold is a $ K- $contact manifold, but the converse holds only if $ dimM^{(2n+1)}=3 $.
On the other hand $ M $ is a Sasakian manifold if and only if one of the following conditions are satisfied; 
\begin{eqnarray}
 \left( {\nabla }_{X_{1}}\phi \right) X_{2}&=&g(X_{1},X_{2})\xi -\eta
 (X_{2})X_{1}\label{Sasakicond1nablaXphi}\\
 R(X_{1},X_{2})\xi &=&\eta (X_{1})X_{2}-\eta (X_{2})X_{1}\label{Sasakicond1R(X,Y)xi}
\end{eqnarray}
for all $X_{1},X_{2}\in \Gamma (TM)$. \par Such as holomorphic sectional curvature of complex manifolds we have $ \phi- $sectional curvature in contact geometry. A Sasakian manifold is called Sasakian space form if the $ \phi- $sectional curvature is constant. Alegre at al. \cite{alegre2004generalized} generalized the Sasakian space forms as in complex space forms. An almost contact metric manifold is
called a generalized Sasakian space form if its curvature has following
form; 
\begin{align}\label{Sasakispaceform}
R(X_{1},X_{2})X_{3}& =F_{1}[g(X_{2},X_{3})X_{1}-g(X_{1},X_{3})X_{2}] \\
& +F_{2}[g(X_{1},\phi X_{3})\phi X_{2}-g(X_{1},\phi X_{3})\phi
X_{2}+2g(X_{1},\phi X_{2})\phi X_{3}] \notag\\
& +F_{3}[\eta (X_{1})\eta (X_{3})X_{2}-\eta (X_{2})\eta (X_{3})X_{1}\notag +g(X_{1},X_{3})\eta (X_{2})\xi-g(X_{2},X_{3})\eta (X_{1})\xi] \notag
\end{align}%
where $F_{1},F_{2}$ and $F_{3}$ are real valued functions on $M$. This type of manifolds contain real space forms and
some classes of contact space forms with special values of $F_{i},\ \
i=1,2,3 $. \par

Let $ M $ be a \NKCMM. \ Then for all $ X_1,X_2 $ vector fields on $ M $ we have following properties \cite{blair1995contact,blair2005concircular}; 
\begin{eqnarray}
\left( {\nabla }_{X_{1}}\phi \right) X_{2}&=&g(X_{1}+hX_{1},X_{2})\xi -\eta
(X_{2})(X_{1}+hX_{1}),\label{nablaXphi}\\
h^{2}&=&(\kappa-1)\phi ^{2},\label{h^2}\\
(\nabla_{X_1}h)X_2&=&[(1-\kappa)g(X_1,\phi X_2)+g(X_1,h\phi X_2)]\xi,+\eta(X_2)[h(\phi X_1+\phi hX_1)]\label{nablaXh}\\
(\nabla_{X_{1}}\eta)X_2&=&g(X_1+hX_1,\phi X_2)\label{nablaXeta}
\end{eqnarray}
 Also we have 
\begin{eqnarray}
R(X_{1},\xi)\xi &=&\kappa [X_{1}-\eta(X_1)\xi] \label{R(X,xi)xi}\\
R(X_{1},X_{2})\xi &=&\kappa\left[ \eta (X_{2})X_{1}-\eta (X_{1})X_{2}\right],\label{R(X,Y)xi}\\
R(X_{1},\xi )X_{2}&=&-\kappa\left[ g(X_{1},X_{2})\xi -\eta (X_{2})X_{1}\right]\label{R(X,xi)Y}. 
\end{eqnarray}
The Ricci curvature  $ S $ and scalar curvature of $ M $ is given by; 
\begin{eqnarray}
 S(X_{1},X_{2})&=&2(n-1)g(X_{1},X_{2})+2(n-1)g(hX_{1},X_{2})\label{RicciX,Y}\\&&+\left[
2n\kappa-2(n-1)\right] \eta (X_{1})\eta (X_{2})\notag\\
S(X_1,\xi)&=&2\kappa n\eta(X_1), S(\xi,\xi)=2\kappa  n\label{RicX,xiandRicxi,xi}\\
\tau&=&2n(2n-2+\kappa)\label{nablascalarcurv}
\end{eqnarray}
for all $ X_{1},X_{2} \in \Gamma(TM)$. \par
We will use the following basic equalities from Riemann geometry; 
\begin{eqnarray}
&&g(X_1,X_2)=\sum_{i=1}^{2n+1}g(X_1,E_i)g(E_i,X_2),\label{g(X,Y)inbasis}\\
&&g(X,Y)=\sum_{i=1}^{2n+1}g(X_1,\phi E_i)g(\phi E_i,X_2),\label{g(phiX,phiY)inbasis}\\
&&\sum_{i=1}^{2n+1}g(hE_i,Ei)=0\label{g(hEi,Ei)}
\end{eqnarray}
where  $ E_i\in \{e_1,...,e_n,\phi e_1,...,\phi e_n,\xi\} $ is the orthonormal basis of $ M $ and $ X_1,X_2 \in \Gamma(TM) $. 
In \cite{blair1995contact}, Blair et al. showed that $ (\kappa, \mu)- $nullity distribution is invariant under 
\begin{equation}\label{barkappavebarmu}
\bar{\kappa}=\frac{\kappa+a^2-1}{a}, \ \  \bar{\mu}=\frac{\mu+2c-2}{a}.
\end{equation}
In \cite{EBoeck} Boeckx introduced the number $ I_{M}=\frac{1-\frac{\mu}{2}}{\sqrt{1-k}} $ for non-Sasakian $(k,\mu)$-contact manifold. This number is called by Boeckx invariant. There are two classes in the classification of non-Sasakian $ (\kappa,\mu)- $spaces. The first class is a manifold with constant sectional curvature $ c $. In this case $ \kappa=c(2-c) $ and $ \mu=-2c $ and by this we get an example of \NKCMM. \  The second class is on $ 3- $dimensional  Lie groups. Boeckx proved that  two Boeckx invariant of two non-Sasakian $ (\kappa,\mu)- $space are equal if and only if this manifolds  are locally isometric as contact metric manifolds.  
Blair, Kim and Tripathi \cite{blair2005concircular} gave following example of \NKCMM s\  by using the Boeckx invariant for the first class. 
\begin{example}\label{example1}
The Boeckx invariant for a $N(1-\frac{1}{n},0)$-manifold is $\sqrt{n}%
	>-1$.  By consider the tangent sphere bundle of an $(n+1)$-dimensional
	manifold of constant curvature $c$, as the resulting $D$%
	-homothetic deformation is $\kappa=c(2-c)$, $ \mu=-2c $ and from ( \ref{barkappavebarmu} ) we get 
	\begin{equation*}
	c=\frac{(\sqrt{n}\pm 1)^2}{n-1},\text{ \ \ \ \ \ \ }a=1+c.
	\textsl{}	\end{equation*}%
	and taking $c$ and $a$ to be these values we obtain $N(1-\frac{1}{n})$%
	-contact metric manifold.
	\end{example}
Blair et al. \cite{blair2005concircular} proved that a \NKCMM\  is locally isometric to Example 1 if the $ \mathcal{Z}(\xi,X). \mathcal{Z}=0 $ for the concircular curvature $ \mathcal{Z} $. Also De et al. \cite{de2014class} showed that $ \xi- $ concircularly flat \NKCMM\ is locally isometric to Example-1. \par 
Blair et al. \cite{blair1995contact} classified $ 3- $dimensional $ (\kappa, \mu)-$spaces by Lie groups of $ 3- $dimensional Riemann manifolds. They gave an example of  $ 3- $dimensional $ (\kappa, \mu)-$space. By taking $ \kappa \leq 1 $ and $ \mu=0 $ in their example we get an example of \NKCMM \ with $ \kappa=1-\lambda ^2 $ for real constant $ \lambda $.  Also De et al. \cite{de2014class} examined the example and they obtained some curvature properties. This example is given as follow. 

\begin{example}	
	Let $M=\{(x_1,x_2,x_3)\in \mathbb{R}^3\}$ be a subset of $\mathbb{R}^3 $,
	where $x_1,x_2,x_3 $ are standard coordinates in $\mathbb{R}^3 $  and $%
	E_1,E_2,E_3 $ be 3-vector fields in $\mathbb{R}^3 $ satisfies 
	\begin{eqnarray*}
		&[E_1,E_2]=(1-\lambda) E_3, \ [E_2,E_3]=2E_1\ \ \text{and}\ \ 
		[E_3,E_1]=(1-\lambda)E_2, \\
		&g(E_1,E_3)=g(E_2,E_3)=g(E_1,E_2)=0, \ \ g(E_1,E_1)=g(E_2,E_2)=1, \\
		& \eta(U)=g(U,E_1),
	\end{eqnarray*}
	where $\lambda $ is a real constant, $g $ is a Riemann metric and $U $ is
	arbitrary vector field on $M $. Take a $(1,1)- $tensor field $ \phi $ is
	defined by 
	\begin{equation*}
	\phi E_1=0, \ \phi E_2=E_3, \ \phi E_3=-E_2.
	\end{equation*}
	Using the linearity of $\phi $ and $g $ we have 
	\begin{eqnarray*}
	\eta(E_1)=1, \ \ \phi^2(U)=-U+\eta(U)E_1
	\end{eqnarray*}
	and 
	\begin{equation*}
	g(\phi X_1,\phi X_2)=g(X_1,X_2)-\eta(X_1)\eta(X_2)
	\end{equation*}
	for any $X_1, X_2 \in \Gamma(TM) $. Moreover, 
	\begin{equation*}
	hE_1=0, \ hE_2=\lambda E_2, \ \text{and} \ hE_3=-\lambda E_3.
	\end{equation*}
	In \cite{de2014class} it is showed that $(M,\phi,\eta,g) $ is a $N(1-{\lambda%
	}^2) $-contact metric manifold. The covariant derivation of orthonormal basis $ \{E_1,E_2,E_3\} $ is given as follow \cite{de2014class} :
 \begin{eqnarray}\label{example2deriv.}
 \nabla_{E_1}E_1=\nabla_{E_1}E_2=\nabla_{E_1}E_3=\nabla_{E_2}E_2=\nabla_{E_3}E_3=0\\
 \nabla_{E_2}E_1=-(1+\lambda)E_3,\ \ , \nabla_{E_2}E_3=(1+\lambda)E_1,\ \ \notag\\
 \nabla_{E_3}E_1=(1-\lambda)E_2, \ \ \nabla_{E_3}E_2=-(1-\lambda)E_1\notag.
 \end{eqnarray}	
\end{example}

\section{\NKCMM s with a Generalized Tanaka Webster Connection}

\subsection{General results}
\begin{definition}[\cite{tanno1989variational}]
	Let $ (M,\phi,\eta,\xi,g) $ be an almost contact metric manifold and $ \nabla $ be a Levi-Civita connection on $ M $. 	For any vector fields $ X_1,X_2 \in \Gamma(TM) $ the following map is called \GTNWC\ on $ M $:  
	\begin{eqnarray*}
		\mathring{\nabla}: \Gamma(TM) \times \Gamma(TM) \rightarrow \Gamma(TM)
		\end{eqnarray*}
	\begin{equation*}\label{TNKdefinalmostcontact}
	\mathring{\nabla}_{X_{1}}X_{2}=\nabla _{X_{1}}X_{2}+(\nabla_{X_{1}}\eta)X_2.\xi-\eta(X_2)\nabla_X\xi+\eta(X_1)\phi Y.
	\end{equation*}
 
\end{definition}

Then from  (\ref{nablaXksi}) and (\ref{nablaXeta}), the \GTNWC \ on  a \NKCMM\  is given by   
\begin{equation} \label{TNKWdef}
\mathring{\nabla}_{X_{1}}X_{2}=\nabla _{X_{1}}X_{2}+g(X_{1}+hX_{1},\phi
X_{2})\xi+\eta (X_{1})\phi
X_{2} +\eta (X_{2})\phi (hX_{1}+X_{1})
\end{equation}
for all $ X_1, X_2 \in \Gamma (TM) $ , where $\nabla $ is Levi-Civita connection on $M $. It is easy to verify that   
$\mathring{\nabla} $ is a linear connection. Also we get  $ (\mathring{\nabla}_{X_{1}}g)(X_{2},X_{3})=0 $, $\mathring{\nabla} $.  On the other hand the torsion of $\mathring{\nabla} $ is given by 
\begin{eqnarray*}
	\mathring{T} &=&(g(X_{1}+hX_{1},\phi X_{2})
	-g(X_{2}+hX_{2},\phi X_{1})\xi +\eta (X_{2})(\phi X_{1}+\phi hX_{1}) \\
	&&-\eta (X_{1})(\phi X_{2}+\phi hX_{2})
\end{eqnarray*}
for every $X_1,X_2 \in \Gamma(TM) $. As a result we get: 

\begin{lemma}
	The map is given by (\ref{TNKWdef}) is a linear, metric and non-symmetric connection on $ M $.  
\end{lemma}
Using  (\ref{nablaXksi}),(\ref{nablaXphi}),   (\ref{nablaXh}), (\ref{nablaXeta}), (\ref{TNKWdef}) and with some computations, for a \NKCMM\ $ M $ with \GTNWC\  we have; 
\begin{eqnarray}
\mathring{\nabla}_{X_{1}}\xi&=&0, \ \ \mathring{\nabla}_{\xi}\xi =0 \label{NewnablaXksi}\\
(\mathring{\nabla}_{X_{1}}\eta )X_{2}&=&0\label{NewnalaXeta},\\
(\mathring{\nabla}_{X_{1}}\phi )X_{2}&=&(\nabla _{X_{1}}\phi)
X_{2}-g(X_{1}+hX_{1},X_{2})\xi +\eta (X_{2})hX_{1}+\eta (X_{2})X_{1}
\label{NewnablaXphi},\\
(\mathring{\nabla}_{X_{1}}h)X_{2} &=&[(\kappa-1)g(\phi X_1,X_2)+g(h X_1,\phi X_2)]\xi   \label{NewnablaXh}\\&&+ \eta(X_1)\phi (X_2+hX_2)\notag 
\end{eqnarray}
for all $ X_1,X_2 \in \Gamma(TM) $

Thus from (\ref{Sasakicond1nablaXphi}) and (\ref{NewnablaXphi}) we get following
corollary:

\begin{corollary}
Let $M $ be a \NKCMM \ with a generalized
Tanaka-Webster connection. If $M $ is Sasakian then $(\mathring{\nabla}%
_{X_{1}}\phi )X_{2}=0 $, for all $X_1, X_2\in \Gamma(TM) $.
\end{corollary}
\begin{remark}
	In \cite{de2016generalized} De and Ghosh proved that on a Sasakian manifold M with generalized Tanaka-Webster connection $ \mathring{\nabla} $ we have $ (\mathring{\nabla}_{X_{1}}\phi )X_{2}=0$ for all $ X_1,X_2 \in \Gamma(TM) $. Thus above corollary is compatible with this result. 
\end{remark}

\subsection{Curvature Properties}

Let $M $ be a \NKCMM\ with generalized
Tanaka-Webster connection. The Riemannian curvature of $M$ is given by 
\begin{equation}\label{Riemanncurvdef.}
\mathring{R}(X_{1},X_{2})X_{3}={\mathring{\nabla}}_{X_{1}}{\mathring{\nabla}}_{X_{2}}{X_3}-{\mathring{\nabla}}_{X_{2}}{\mathring{\nabla}}_{X_{1}}{X_3}-{\mathring{\nabla}}_{[X_{1},X_2]}{X_3}
\end{equation}
By using (\ref{TNKWdef}), from (\ref{nablaXphi}), (\ref{nablaXh}) and with a long computation we get 
%\begin{eqnarray}
%{\mathring{\nabla}}_{X_{1}}{\mathring{\nabla}}_{X_{2}}{X_3}=
%\end{eqnarray}
%\begin{eqnarray}
%{\mathring{\nabla}}_{X_{2}}{\mathring{\nabla}}_{X_{1}}{X_3}=
%\end{eqnarray}
%\begin{eqnarray}
%{\mathring{\nabla}}_{[X_{1},X_2]}{X_3}=
%\end{eqnarray}
\begin{eqnarray}\label{TNKRiemanCRV}
\mathring{R}(X_{1},X_{2})X_{3} &&=R(X_1,X_2)X_3+\kappa \{(\eta (X_{2})g(X_1,X_3)-\eta(X_1)
 g(X_2,X_3))\xi\\&& -\eta(X_2)\eta (X_{3})X_1+\eta(X_1)\eta (X_{3})X_2 \}\notag\\&& 
-g(X_2+hX_2,\phi X_3)[\phi X_1+\phi hX_1]
+g(X_1+hX_1,\phi X_3)[\phi X_2+\phi hX_2]\notag\\&&
+[g(X_1,\phi X_2+\phi hX_2)+g(X_2,\phi X_1+ \phi hX_1)]\phi X_3\notag\
\end{eqnarray}
where $ R $ is the Riemann curvature of $ M $ with Levi-civita connection and  $ \mathring{R} $ is the Riemann curvature of $ M $ with \GTNWC. 
The symmetry properties of $\mathring{R} $ are given by 
\begin{eqnarray*}
	\mathring{R}(X_{1},X_{2},X_{3},X_{4})+\mathring{R}(X_{2},X_{1},X_{4},X_{3})&&=0\\
	\mathring{R}(X_{1},X_{2},X_{3},X_{4})+\mathring{R}(X_{1},X_{2},X_{4},X_{3})&&=0\\
	\mathring{R}(X_{1},X_{2},X_{3},X_{4})+\mathring{R}(X_{3},X_{4},X_{1},X_{2})&& =-2(g(\phi  X_1,X_4)g(hX_2,\phi X_3)-g(X_2,\phi X_3)g(\phi hX_1,X_4)\\&&-g(hX_1,\phi X_3)g(X_2,\phi X_4)+g(X_1,\phi X_3)g(\phi hX_2,X_4)\\&&-g(X_3,\phi h X_4)g(\phi X_1,X_2))
\end{eqnarray*}
Also we have 
\begin{eqnarray*}
\mathring{R}(X_{1},X_{2})X_{3}+\mathring{R}(X_{2},X_{3})X_{1}+\mathring{R}
(X_{3},X_{1})X_{2}&=&2( g(\phi X_{1},X_{2})\phi hX_{3}\\&&-g(\phi
X_{1},X_{3})\phi hX_{2}+g(\phi X_{2},X_{3})\phi hX_{1}).
\end{eqnarray*}
It is easy to see that Bianchi identity of $\mathring{R}$ is satisfied when $\xi $ is
Killing. \par 
On the other hand from (\ref{R(X,xi)xi}),(\ref{R(X,Y)xi}), (\ref{R(X,xi)Y}) and (\ref{TNKRiemanCRV}) we get 
\begin{eqnarray}\label{NewReq.}
\mathring{R}(X_{1},X_{2})\xi =
\mathring{R}(\xi ,X_{1})X_{2}=
\mathring{R}(X_{1},\xi )\xi =0.
\end{eqnarray}

A generalized Sasakian space form is a class of almost contact metric manifolds
which is defined on an almost contact metric manifold by the curvature relation (\ref{Sasakispaceform}). In the
following theorem we obtain a new example of generalized Sasakian space
forms.

\begin{theorem}
Let $M $ be a \NKCMM\ with generalized
Tanaka-Webster connection\ . If $\xi $ is Killing vector field then $M $ is a
generalized Sasakian space forms with $F_{1}=F_{3}=\kappa,\ F_{2}=1 $.
\end{theorem}

\begin{proof}
	Let $ M $ be a \NKCMM\ with \GTNWC\ . If $ \xi $ is Killing then from (\ref{TNKRiemanCRV}) we have
	\begin{eqnarray*}
		\mathring{R}(X_{1},X_{2})X_{3} &=&\kappa \left\{
		g(X_{2},X_{3})X_{1}-g(X_{1},X_{3})X_{2}\right\} \\
		&&+g(X_{1},\phi X_{3})\phi X_{2}-g(X_{2},\phi X_{3})\phi X_{1}+2g(X_{1},\phi
		X_{2})\phi X_{3} \\
		&&+\kappa \left\{ \eta (X_{1})\eta (X_{3})X_{2}-\eta (X_{2})\eta
		(X_{3})X_{1}+g(X_{1},X_{3})\eta (X_{2})\xi -g(X_{2},X_{3})\eta (X_{1})\xi
		\right\}.
	\end{eqnarray*}
 This shows us $ M $ is generalized Sasakian space forms with $F_{1}=F_{3}=\kappa ,\ F_{2}=1 $.
\end{proof}
The Ricci curvature of a \NKCMM\ with \GTNWC \ is defined by 
\begin{equation*}
\mathring{S}(X_1,X_4)=\sum_{i=1}^{2n+1}\mathring{R}(X_1,E_i,E_i,X_4)
\end{equation*}
where $ E_i , 1\leq i \leq 2n+1$ are the orthonormal basis of $ M $ and $ X_1,X_4 \in \Gamma(TM).  $ Thus from (\ref{g(phiX,phiY)}),(\ref{g(X,Y)inbasis}), (\ref{g(phiX,phiY)inbasis}), (\ref{g(hEi,Ei)}) and (\ref{TNKWdef}) we have 
\begin{eqnarray}
\mathring{S}(X_1,X_4)&&=S(X_1,X_4)+(3-\kappa)g(X_1,X_4)\\&&+(-(2n-1)\kappa-3)\eta(X_1)\eta(X_4)-g(hX_1,hX_4). \notag
\end{eqnarray}
On the other hand since $ h $ is symmetric from (\ref{h^2}) we get 
\begin{equation*}
g(hX_1,hX_4)=(\kappa-1)(-g(X_1,X_4)+
\eta(X_1)\eta(X_4)).
\end{equation*}
Thus the Ricci curvature of a \NKCMM\ with \GTNWC \ is obtained as  
\begin{equation}\label{NewRicci}
\mathring{S}(X_1,X_4)=S(X_1,X_4)+2g(X_1,X_4)-2(n\kappa+1)\eta(X_1)\eta(X_4). 
\end{equation}

and so, from (\ref{RicciX,Y}) we get 
\begin{equation*}
\mathring{S}(X_1,X_4)=2ng(X_{1},X_{2})+2(n-1)g(hX_{1},X_{2})-2n \eta (X_{1})\eta (X_{2}).
\end{equation*}

 Also we have 
\begin{equation}\label{NewRic(X,xi)}
\mathring{S}(X_{1},\xi )=\mathring{S}(\xi,\xi )=0.
\end{equation}

 From (\ref{RicciX,Y}) we know if a \NKCMM\ is Sasakian then it is $ \eta- $Einstein. 
 \begin{corollary}
 	If a \NKCMM\ with Levi Civita connection $ M $ is Sasakian then it is $ \eta- $Einstein with \GTNWC.  
 \end{corollary}
\begin{proof}
	Let $ M $ be a \NKCMM\ with Levi-Civita connection. If $ M $ is Sasakian then from (\ref{RicciX,Y}) we get 
	\begin{equation*}
	S(X_1,X_4)=2(n-1)g(X_1,X_4)+2\eta(X_1)\eta(X_4).
	\end{equation*} 
	Thus from (\ref{NewRicci}) we obtain 
	\begin{equation*}
	\mathring{S}(X_1,X_4)=2ng(X_1,X_4)-2n\eta(X_1)\eta(X_4)
	\end{equation*}
	which shows us $ M $ with \GTNWC\ is $ \eta- $Einstein.
\end{proof}
The scalar curvature of a \NKCMM\ with \GTNWC\ $ M $ is obtained as 
\begin{equation*}
\mathring{\tau}=\tau+4n-2n\kappa.
\end{equation*}
From (\ref{nablascalarcurv}) we get 
\begin{equation}\label{Newnablascalar}
\mathring{\tau}=4n^2.
\end{equation}

\section{Concircular Curvature Tensor on $ N(\kappa)-$Contact Metric Manifold with Generalized Tanaka-Webster Connection }
 The concircular curvature was defined by Yano \cite{yano1940concircular}. Blair et al. \cite {blair2005concircular,de2014class} studied on 
 \NKCMM\  under certain curvature conditions via concircular curvature tensor.  In this section we study on concircular curvature tensor on a \NKCMM\ $M $ with generalized Tanaka-Webster connection. \par 
 Concircular curvature tensor $ \mathcal{Z} $ on \NKCMM\ with \GTNWC\ is given by
\begin{eqnarray*}
\mathcal{\mathring{Z}}(X_{1},X_{2})X_{3}=\mathring{R}%
(X_{1},X_{2})X_{3}-\frac{\mathring{\tau}}{2n(2n+1)}[g(X_{2},X_{3})X_{1}-g(X_1,X_{3})X_{2}].
\end{eqnarray*}
From (\ref{Newnablascalar}) we get 
\begin{eqnarray} \label{NewZ}
	\mathcal{\mathring{Z}}(X_{1},X_{2})X_{3}=\mathring{R}%
	(X_{1},X_{2})X_{3}-\frac{2n}{2n+1}[g(X_{2},X_{3})X_{1}-g(X_1,X_{3})X_{2}].
\end{eqnarray}
For all $X_1, X_2, X_3 \in \Gamma(TM) $  we obtain 
\begin{eqnarray}
&&\mathcal{\mathring{Z}}(X_{1},\xi )\xi= K\phi^2X_1\label{NewZ(X,xi)xi}\\&&
\mathcal{\mathring{Z}}(X_{1},X_{2})\xi = K(\eta(X_2)X_1-\eta(X_1)X_2)\label{NewZ(X,Y)xi}\\
&& 
\mathcal{\mathring{Z}}(X_{1},\xi )X_{2}= K(\eta(X_2)X_1-g(X_1,X_2)\xi)\label{NewZ(X,xi)Y}\\&&
\eta(\mathcal{\mathring{Z}}(X_{1},X_{2})X_{3})=K(\eta(X_3)g(X_1,X_2)-\eta(X_1)g(X_3,X_2))\label{etanNewZ}
\end{eqnarray}
where $ K=-\frac{2n}{2n+1} $. \par

A \NKCMM\ with Levi-Civita conectiom is called $\xi- $concircularly flat if $\mathcal{Z}%
(X_1,X_2)\xi=0 $.  In \cite{de2014class} De et al. proved that $%
\xi- $concircularly flat \NKCMM \ is locally isometric
to Example \ref{example1}. We recall a \NKCMM\ with \GTNWC\   by $\mathring{\xi}- $concircularly flat if $\mathring{\mathcal{Z}}
(X_1,X_2)\xi=0 $. Thus from (\ref{NewZ(X,Y)xi}) we obtain: 
\begin{theorem}
	A \NKCMM\ with \GTNWC\  can not to be $\mathring{\xi}- $concircularly flat.
\end{theorem}
A \NKCMM\ is called $ \eta- $Einstein if $ S(X_1,X_2)=Ag(X_1,X_2)+B\eta(X_1)\eta(X_2) $ for smooth functions $ A $ and $ B $  on the manifold. $ \eta- $Einstein manifolds are generalization of Einstein manifolds which arises from the general relativity. In the following result we give a condition for a \NKCMM \ with \GTNWC\ to be $ \eta- $Einstein. \par   
A \NKCMM\ is called $ \phi- $concircularly flat with \GTNWC\ if $ g(\mathcal{\mathring{Z}}(\phi X_1,\phi X_2)\phi X_3,\phi X_4)=0$.

\begin{theorem}
A $ \phi- $concircularly flat \NKCMM \ with \GTNWC\ is $ \eta -$Einstein. 
\end{theorem}
\begin{proof}
Let $ M $ be a $ \phi- $concircularly flat  \NKCMM\  with the  \GTNWC\ .  Thus from (\ref{NewZ}) we have
\begin{eqnarray}\label{R(phiX,..)}
g(\mathring{R}(\phi X_1,\phi X_2)\phi X_3,\phi X_4)=\frac{2n}{2n+1}(g(\phi X_2,\phi X_3)g(\phi X_1, \phi X_4)-g(\phi X_2, \phi X_4)).
\end{eqnarray}
Let $ \{e_1,...,e_n, \phi e_1,...,\phi e_n, \xi\} $ be an orthonormal $ \phi $-basis of the tangent space. Putting $ X_2=X_3=E_i $ in (\ref{R(phiX,..)})  and by taking summation over $ i=1 $ to $ 2n+1 $ we get 
\begin{equation*}
\mathring{S}(X_1,X_4)=\frac{2n(2n-1)}{2n+1}g(\phi X_1,\phi X_4).
\end{equation*}
Replacing $ X_1 $ and $ X_4 $ by $ \phi X_1 $ and $ \phi X_4 $ and using (\ref{NewRicci}), (\ref{NewRic(X,xi)}) we obtain
\begin{equation*}
S(X_1,X_4)=(\frac{-2(2n^2+n+1)}{2n+1})g(X_1,X_4)+(\frac{2n(2n-1)}{2n+1}+2(n\kappa +1))\eta(X_1)\eta(X_4).
\end{equation*}
Thus $ M $ is $ \eta- $Einstein. 

\end{proof}
A Riemann manifold $ M $ is called locally symmetric or semi-symmetric if $ R.R=0 $. Also if $ R.S=0 $ then $ M $ is called Ricci semi-symmetric manifold. For two $ (1,3) -$type tensors $ \mathcal{T}_1,\mathcal{T}_2 $ we have 
\begin{eqnarray}\label{generaltensorproduct}
	(\mathcal{T}_1(X_1,X_2).\mathcal{T}_2)(X_3,X_4)X_5&=&\mathcal{T}_1(X_1,X_2)\mathcal{T}_2(X_3,X_4)X_5-\mathcal{T}_2(\mathcal{T}_1(X_1,X_2)X_3,X_4)X_5\\&&-\mathcal{T}_2(X_3,\mathcal{T}_1(X_1,X_2)X_4)X_5-\mathcal{T}_2(X_3,X_4)\mathcal{T}_1(X_1,X_2)X_5\notag 
\end{eqnarray}
and for $ (0,2 )-$type tensor $ \mathcal{\omega} $ we have 
\begin{equation}\label{generalT.S}
(\mathcal{T}_1(X_1,X_2).\mathcal{\omega})(X_3,X_4)=\mathcal{\omega}(\mathcal{T}_1(X_1,X_2)X_3,X_4)+\mathcal{\omega}(X_3,\mathcal{T}_1(X_1,X_2)X_4).
\end{equation}
In \cite{blair2005concircular} the authors proved that a \NKCMM\  which satisfies $\mathcal{Z}(\xi ,X)\mathcal{Z}=0$ is locally
isometric to $\mathbb{E}^{n+1}\times\mathbb{R}^n $. Also a \NKCMM\ is $\mathcal{Z}(\xi ,X)S=0$  if and only if the manifold is an Einstein-Sasakian
manifold. We consider a \NKCMM \ with
generalized Tanaka-Webster connection under the conditions $\mathcal{\mathring{Z}}%
(\xi ,X)\mathcal{\mathring{Z}}=0$ and $\mathcal{\mathring{Z}}%
(\xi ,X){\mathring{S}}=0$  for all $ X \in \Gamma(TM) $.
\begin{theorem}
	A \NKCMM \ with \GTNWC\ can not satisfy $\mathcal{\mathring{Z}}%
	(\xi ,X){\mathring{S}}=0$. 
\end{theorem}
\begin{proof}
	Let $ M $ be a \NKCMM \ with \GTNWC\  which is satisfied the condition  $\mathcal{\mathring{Z}}%
	(\xi ,X_1){\mathring{S}}=0$. Then from (\ref{generalT.S}) we get 
	\begin{equation*}
	\mathring{S}(\mathcal{\mathring{Z}}(\xi ,X_1)X_2,X_3)+	\mathring{S}(X_2,\mathcal{\mathring{Z}}(\xi,X_1)X_3)=0. 
	\end{equation*}
	Let take $ X_3=\xi $ then by using (\ref{NewRic(X,xi)}) and  (\ref{NewZ(X,xi)Y}) we obtain 
	\begin{equation*}
	K\mathring{S}(X_1,X_2)=0. 
	\end{equation*}
	This completes the poof. 
\end{proof}
\begin{theorem}
A \NKCMM \ with \GTNWC\ can not satisfy $\mathcal{\mathring{Z}}(\xi ,X)\mathcal{\mathring{Z}}=0$. 
\end{theorem}

\begin{proof}
	From (\ref{generaltensorproduct}) we have
	\begin{eqnarray*}
		(\mathcal{\mathring{Z}}_1(\xi,X_2).\mathcal{\mathring{Z}}_2)(X_3,X_4)X_5&=&\mathcal{\mathring{Z}}_1(\xi,X_2)\mathcal{\mathring{Z}}_2(X_3,X_4)X_5-\mathcal{\mathring{Z}}_2(\mathcal{\mathring{Z}}_1(\xi,X_2)X_3,X_4)X_5\\&&-\mathcal{\mathring{Z}}_2(X_3,\mathcal{\mathring{Z}}_1(\xi,X_2)X_4)X_5-\mathcal{\mathring{Z}}_2(X_3,X_4)\mathcal{\mathring{Z}}_1(\xi,X_2)X_5\notag 
	\end{eqnarray*}
	for all $ X_2,X_3, X_4, X_5\in \Gamma(TM) $. 
		Suppose that $\mathcal{\mathring{Z}}(\xi ,X_2)\mathcal{\mathring{Z}}=0$. Let take $ X_5=\xi $  then from (\ref{NewZ(X,Y)xi}), (\ref{NewZ(X,xi)Y}) and with a long computations we get 	
\begin{eqnarray*}
\mathcal{\mathring{Z}}(X_2,X_3)X_1&=&K\{[2(\eta(X_3)g(X_1,X_3)-\eta(X_2)g(X_1,X_3))]\xi\\&&-g(\phi X_1,\phi X_2)X_3-g(\phi X_1,\phi X_3)X_2\}
\end{eqnarray*}
By setting $ X_3=\xi $ we obtain
\begin{equation*}
\eta(\mathcal{\mathring{Z}}(X_2,\xi)X_1)=-3Kg(\phi X_2,\phi X_1).
\end{equation*}
Thus from (\ref{NewZ(X,xi)Y}) we have $ g(\phi X_1, \phi X_2)=0. $ So the condition $\mathcal{\mathring{Z}}(\xi ,X_2)\mathcal{\mathring{Z}}=0$ can not satisfy. 
	\end{proof}

\section{Example}

Let $ M $ be a \NKCMM\ which is given in Example 2 with a  \GTNWC. Then from (\ref{example2deriv.}) and by using  (\ref{TNKWdef}) we get 
 \begin{eqnarray*}
&&\mathring{\nabla}_{E_1}E_1=\mathring{\nabla}_{E_2}E_1=\mathring{\nabla}_{E_3}E_1=\mathring{\nabla}_{E_2}E_2=\mathring{\nabla}_{E_2}E_3=\mathring{\nabla}_{E_3}E_2=\mathring{\nabla}_{E_3}E_3=0\\
&&\mathring{\nabla}_{E_1}E_2= E_3, \  \mathring{\nabla}_{E_1}E_3=-E_2.
\end{eqnarray*}
Thus from (\ref{Riemanncurvdef.}) the curvature of $ M $ is obtained as followings:  
\begin{eqnarray*}
&&\mathring{R}_{121}=\mathring{R}_{122}=\mathring{R}_{123}=\mathring{R}_{131}=\mathring{R}_{132}=\mathring{R}_{133}=\mathring{R}_{231}=0,\\
&&\mathring{R}_{232}=-2E_3,\mathring{R}_{233}=2E_2 \ \ 
\end{eqnarray*}
where $ \mathring{R}_{ijk}=\mathring{R}(E_i,E_j)E_k $. Also we can obtain same results from (\ref{TNKRiemanCRV}). 
By using the definition of \GTNWC\ , we get $ \mathring{\nabla}_{E_i}E_1=0 $, $ (\mathring{\nabla}_{E_i}\phi)E_j=0 $ and $  (\mathring{\nabla}_{E_i}\eta)E_j=0$ for $ 1\leq i,j \leq 3.  $\par 
The Ricci and scalar curvature of $ M $ is obtained by 
\begin{eqnarray*}
\mathring{S}(E_1,E_1)=0, \ \mathring{S}(E_2,E_2)=2, \ \mathring{S}(E_3,E_3)=2
\end{eqnarray*}
and thus $ \mathring{\tau}=4. $ These results verify (\ref{NewRicci}) and (\ref{Newnablascalar}).
Suppose that the condition $ \mathring{Z}(E_1,E_j).\mathring{S}=0 $ is satisfied on  $ M $ . Then from (\ref{generalT.S}) we get 
\begin{equation*}
\mathring{S}(\mathcal{\mathring{Z}}(E_1 ,E_j)E_k,E_r)+	\mathring{S}(E_k,\mathcal{\mathring{Z}}(E_1,E_j)E_r)=0. 
\end{equation*}
Let choose $ E_r=E_1 $ then from (\ref{NewRic(X,xi)}), (\ref{NewZ(X,xi)xi}) and  (\ref{NewZ(X,xi)Y}) we obtain 
\begin{equation*}
\frac{2}{3}\mathring{S}(E_k,E_j)=0. 
\end{equation*}
If $ E_k=E_j=E_2$ then since $ \mathring{S}(E_2,E_2)=2 $ there is contradiction. So $ \mathring{Z}(E_1,E_j).\mathring{S}=0 $ can not satisfy on $ M $. This is verified the Theorem 5.

\end{document}